\documentclass[a4paper]{amsart}

\usepackage{amssymb}

\numberwithin{equation}{section}

\newtheoremstyle{ttheorem}%
       {1.8ex\@plus1ex}                
       {2.1ex\@plus1ex\@minus.5ex}      
       {\itshape}           
       {0pt}                   
       {\bfseries}          
       {.}                  
       {.5em}               
       {}                

\newtheoremstyle{ddefinition}%
       {1.8ex\@plus1ex}                
       {2.1ex\@plus1ex\@minus.5ex}      
       {}           
       {0pt}                   
       {\bfseries}           
       {.}                  
       {.5em}               
       {}                

\newtheoremstyle{rremark}%
       {1.8ex\@plus1ex}                
       {2.1ex\@plus1ex\@minus.5ex}      
       {\normalfont}        
       {0pt}                   
       {\bfseries}           
       {.}                  
       {.5em}               
       {}                   

\theoremstyle{ttheorem}
\newtheorem{theorem}{Theorem}[section]
\newtheorem{lemma}[theorem]{Lemma}
\newtheorem{proposition}[theorem]{Proposition}

\theoremstyle{ddefinition}
\newtheorem{definition}[theorem]{Definition}

\theoremstyle{rremark}
\newtheorem{remark}[theorem]{Remark}
\newtheorem{myremarks}[theorem]{Remarks}

\newcommand{\oplam}{\mbox{\Large $\curlywedge$}}

\newenvironment{remarks}{\begin{myremarks}\begin{nummer}}%
    {\end{nummer}\end{myremarks}}
    {\end{nummer}\end{myexamples}}

\newcounter{numcount}
\newcommand{\labelnummer}{\mbox{\normalfont (\roman{numcount})}}%

\makeatletter

\newenvironment{nummer}%
  {\let\curlabelspeicher\@currentlabel%
    \begin{list}{\labelnummer}%
      {\usecounter{numcount}\leftmargin0pt%
        \topsep0.5ex\partopsep2ex\parsep0pt\itemsep0ex\@plus1\p@%
        \labelwidth2.5em\itemindent3.5em\labelsep1em%
      }%
    \let\saveitem\item%
    \def\item{\saveitem%
      \def\@currentlabel{{\upshape\curlabelspeicher}$\,$\labelnummer}}%
    \let\savelabel\label%
    \def\label##1{\savelabel{##1}%
      \@bsphack%
        \ifmmode\else%
          \protected@write\@auxout{}%
          {\string\newlabel{##1item}{{\labelnummer}{\thepage}}}%
        \fi%
      \@esphack%
    }%
  }{\end{list}}%

\begin{document}


\title{On pattern entropy of weak model sets}

\begin{abstract}
We study point sets arising from cut-and-project constructions. 
An important class is that of weak model sets, which include squarefree numbers and visible
lattice points. For such model sets, 
we give a non-trivial upper bound on their pattern entropy in
terms of the volume of the window boundary in internal space. This
proves a conjecture by R.\ V.\ Moody.
\end{abstract}

\author[C.\ Huck]{Christian Huck}
\address{Fakult\"at f\"ur Mathematik,
  Universit\"at Bielefeld,
  Postfach 100131,
  33501 Bielefeld, Germany}
\email{huck@math.uni-bielefeld.de}

\author[C.\ Richard]{Christoph Richard}        
\address{Department f\"ur Mathematik,
  Friedrich-Alexander-Universit\"at Erlangen-N\"urnberg,
  Cauerstra{\ss}e 11,
  91058 Erlangen, Germany}
\email{christoph.richard@fau.de}

\dedicatory{to the memory of Peter A.\ B.\ Pleasants\/\dag}

\maketitle

\section{Introduction}

Baake, Moody and Pleasants \cite{bmp00} gave a cut-and-project construction for the visible points of an $n$-dimensional lattice in Euclidean space and the
\mbox{$k^{\rm th}$-power}-free numbers, with the internal space
adelic, instead of Euclidean as in more usual cut-and-project
sets. This generalised a cut-and-project construction
of squarefree numbers by Meyer \cite{me73}. In these constructions, the boundaries of the windows have positive Haar measure, however, so the corresponding points sets are not regular model sets. In particular,
results about diffraction of regular model sets could not be applied to these point sets, and their pure point diffractivity was shown by explicit calculation in \cite{bmp00, ph13}. We  mention the monograph \cite{bg13} for a modern comprehensive exposition of the subject. Recently, there has been a renewed interest in such non-regular model sets due to their rich combinatorial and dynamical properties, see \cite{cs13,bh14} and references therein.

Unexpectedly, in these examples the density of the cut-and-project set was seen to be equal to the volume of the window. This relationship,
also called the density formula, was known to hold for regular model sets, if the Haar measure of the embedding space is normalised such that the underlying lattice has density one.  Whereas special cases go back to \cite{me72}, the density formula had been proved for all regular model sets by Schlottmann \cite{sch98}. But for the above examples, it was also pointed out in \cite{bmp00} that translating the window can
cause the density of the corresponding model set to change and can, indeed,
make the model set vanish altogether.  
To be more explicit about these windows:
they are closed but their complements are dense in the internal space, and
as a consequence the boundary of a window coincides with the window itself.
Moody \cite{m02} has since proved the surprising result
that for a very general class of cut-and-project sets, which he calls weak model sets, 
 the density of the model set is indeed equal to the measure of the window
for \textit{almost} all translations of the window. Also, in the model set
description of the above examples or, more generally, of the $k$-free
 points of a lattice \cite{p06,ph13}, their pattern entropy turns out to be equal to the measure
of the corresponding window boundary. Since translating the window can result in an
empty model set, this again is not a relationship that can always hold.

Moody \cite{p06} has suggested that the relationship between the
pattern entropy and the window boundary may be akin to that
between the density and the window itself. We will consider this
question for  weak model sets, which have initially been studied by 
Schreiber \cite{schr71, schr73}. As pointed out by Pleasants \cite{p06}, there is a version of the 
density formula which holds for any translation of the window in that situation. (In fact, such a version 
was already used by Meyer  \cite[Rem.~(6.2)]{me73}.)
If $W$ is the window and $\Lambda$ the model set derived from it, then
\begin{displaymath}
\theta_H(\mathrm{int}(W))\le\underline{\mathrm{dens}}(\Lambda)\le\overline{\mathrm{dens}}(\Lambda)\le\theta_H(\mathrm{cl}(W)),
\end{displaymath}
where $\underline{\mathrm{dens}}(\Lambda)$ and $\overline{\mathrm{dens}}(\Lambda)$ are the lower and upper
densities of $\Lambda$, and $\theta_H$ is the unique Haar measure on internal space according to the normalisation explained above. 
So in this more general setting
the density too, when it exists, is confined to an interval determined
by the window, and this formula holds for any translation of the window.
Concerning pattern entropy $h^*(\Lambda)$ of a weak model set, which will be defined below, the measure
of the window boundary gives indeed an upper bound,
\begin{displaymath}
h^*(\Lambda)\le \theta_H(\partial W) \log 2,
\end{displaymath}
for any translate of the window.  As suggested by Pleasants \cite{p06}, both results can be proved 
by approximation with regular model sets using the density formula for regular model sets.

Following this route, we review weak model sets in Section~\ref{sec:wms}. This is done in detail as -- in contrast to (full) model sets -- previous results on weak model sets are somewhat scattered through the literature, using different terminologies. We will prove the extension of the density formula mentioned above in Section~\ref{sec:df}. Then we consider pattern entropy. In Euclidean space, this quantity has previously been studied in \cite{l99,lp03}, where it is called configurational entropy, and in \cite{blr07}, where it is called patch counting entropy. Corresponding complexity measures for model sets have also been studied  in \cite{j10}. We will discuss these approaches.
 We will prove the above entropy bound for weak model sets in Theorem~\ref{thm:main} for a non-compact direct space.  As we will point out, our arguments do not rely on commutativity of the underlying groups, if they are assumed to be second countable. We would also like to stress that our approach is geometric and avoids dynamical systems. There are however important connections between pattern entropy and topological entropy of the so-called hull associated with the model set, which we will indicate below. Section~\ref{sec:sl} specialises to subsets of lattices and briefly discusses examples where the entropy bound is sharp. Our cut-and-project scheme is taken from Sing \cite[Sec.~5a]{sing07}. It avoids adelic internal spaces and thus simplifies previous analyses. 

\section{Weak model sets}\label{sec:wms}

A model set is a projection of a certain lattice subset. We recall the relevant framework following \cite{bg13, m97} and \cite{me72, schr73}. We use the abbreviation LCA group for a locally compact Hausdorff abelian group. In the following we prefer multiplicative notation for the group operation since the arguments and results in this paper do not rely on commutativity. 

\begin{definition}[Cut-and-project scheme]\label{def:cpscheme}
A \emph{cut-and-project scheme} is a triple $(G, H,\mathcal L)$ with $\sigma$-compact LCA groups $G,H$, and a discrete subgroup $\mathcal L$ of $G\times H$ such that $(G\times H)/{\mathcal L}$ is compact.  It is assumed that the canonical projections $\pi_G:G\times H\to G$, $\pi_H:G\times H\to H$ satisfy $\pi_G|_{\mathcal L}$ is one-to-one and $\pi_H(\mathcal L)$ is dense in $H$. We call $G$ \textit{direct space} and $H$ \textit{internal space}.
\end{definition}

\begin{remarks}\label{rem:fdom}
\item Many relevant examples have Euclidean direct and internal space. For weak model sets, which are defined below, one can assume without loss of generality that the internal space is second countable, see Remark~\ref{rem:wms}. Note that in locally compact Hausdorff spaces, second countability is equivalent to $\sigma$-compactness and metrisability.
\item  The discrete group $\mathcal L$ is a lattice in $G\times H$, i.e., $(G\times H)/{\mathcal L}$ admits a non-trivial finite invariant regular Borel measure \cite[Prop.~9.1.5]{de09}. The lattice $\mathcal L$ is countable since $G,H$ are $\sigma$-compact. We fix (left) Haar measures $\theta_G, \theta_H$ on $G,H$ and choose the product measure as Haar measure on $G\times H$. By \cite[Lem.~2]{kk98}, the lattice $\mathcal L$ has measurable relatively compact fundamental domains, whose common finite measure we denote by $1/\mathrm{dens}(\mathcal L)$. In fact, $\mathrm{dens}(\mathcal L)$ is the (canonically defined) density of lattice points in $G\times H$.
\end{remarks}

Writing $L:=\pi_G(\mathcal L)$, the canonical map $\star:L\to H$ is
called the \textit{star map}. We thus have $\mathcal L=\{(\ell,\ell^\star)\,|\,\ell \in L\}$.
Since $L^\star$ is dense in $H$, the lattice $\mathcal L$ admits
fundamental domains within arbitrarily thin strips $G\times U$ with
non-empty open $U$. This is in line with \cite[Lem.~II.10]{me72}, \cite[Lem.~2.5]{m97}, \cite[Lem.~7.4]{bg13}. As it is central for our entropy estimate, we also give its short proof.

\begin{lemma}\label{lem:smallfd}
Let $(G, H,\mathcal L)$ be a cut-and-project scheme. Then for
any non-empty open $U\subset H$ there exists a compact $F\subset G$  satisfying
\begin{displaymath}
(F\times U) \mathcal L=G\times H.
\end{displaymath}
\end{lemma}

\begin{proof}
Consider any non-empty open $U$ in $H$. As $\mathcal L$ is a lattice, we may choose compact $\mathcal F\subset G\times H$ such that $\mathcal F \mathcal L=G\times H$ by Remark~\ref{rem:fdom} (ii). Consider the compact set $A=\pi_G(\mathcal F)$.  Since $\pi_H(\mathcal F)$ is compact and $\pi_H(\mathcal L)$ is dense in $H$, there exist $\ell_1,\ldots, \ell_n\in L$ such that 
\begin{displaymath}
\mathcal F \subset A\times \bigcup_{i=1}^n\ell_i^\star U.
\end{displaymath}
With the compact set $F:=\bigcup_{i=1}^n \ell_i^{-1}A$, we compute $(F\times U) \mathcal L=G\times H$.
\end{proof}

We want to study certain subsets of $L$. For that reason, we fix a \textit{window} $W\subset H$ and define
a \textit{projection set} by
\begin{displaymath}
\oplam(W)=\{\ell \in L\,|\, \ell^\star\in W\}.
\end{displaymath}
Diagrammatically, 
the construction looks like this:
\[\begin{array}{ccccc}
G&\longleftarrow&G\times H &\longrightarrow&H\\
\cup&&\cup&&\cup\\ \oplam(W)&&\mathcal L&&W\\ \uparrow&&\uparrow&&\uparrow\\
\ell&\longleftrightarrow&(\ell,\ell^\star)&\longrightarrow&\ell^\star
\end{array}\]
If the star map were one-to-one (which we could assume without loss of generality \cite[Prop.~4]{schr73}), then any subset of $L$ could be described by some window $W$.

We list some properties of $\oplam(W)$. Recall that $D\subset G$ is \textit{uniformly discrete} if  there is a unit neighbourhood $U\subset G$ such that any of its translates $xU$, where $x\in G$, contains at most one point of $D$.  A set $D\subset G$ is \textit{relatively dense} if there is a compact set $K\subset G$ such that $DK=G$. If $D$ is uniformly discrete and relatively dense, then $D$ is called a \textit{Delone} set.

\begin{proposition}
If $W$ is relatively compact, then $\oplam(W)$ is uniformly discrete. \qed
\end{proposition}

\begin{remark}
See \cite[Prop.~2]{schr73} and \cite[Prop.~2.6]{m97}, \cite[Lem. 2.3.~ii)]{st14}, \cite[Prop.~7.5]{bg13} for proofs.
If $\Lambda\subset L$ is uniformly discrete, then $\Lambda^\star\subset H$ need not be relatively compact. It is straightforward to construct counterexamples within a cut-and-project scheme $(\mathbb R, \mathbb R,\mathcal L)$ with $\mathcal L$ a rotated copy of $\mathbb Z^2$.
\end{remark}

\begin{definition}[Weak model set]
Let $(G, H,\mathcal L)$ be a cut-and-project scheme and take $W\subset H$ relatively compact.  Then the projection set $\oplam(W)$, or any translate $t\oplam(W)$ with $t\in G$, is called a \emph{weak model set}.
\end{definition}

\begin{remarks}
\item Our notion of weak model set differs from that in \cite[Sec.~3]{m02}, where measurability of $W$ is required in addition. Our results below do not require measurability of $W$. 
\item Weak model sets $\oplam(W)$ have initially been studied by Schreiber in \cite{schr71,schr73}, where they are called models \cite[Def.~2]{schr73}. Every subset of a weak model set is a weak model set, possibly with a different internal space \cite[Cor.~2]{schr73}. Every weak model set is harmonious \cite[Thm.~1]{schr73}.  Note that Meyer sets \cite{m97} are relatively dense weak model sets.
\end{remarks}

\begin{remark}
Let $\oplam(W)$ be a weak model set. If  $\mathrm{int}(W)\ne\varnothing$, then $\oplam(W)$ is called a model set \cite[Def.~2.4]{m97}, and also any of its translates is called a model set \cite[Def.~7.2]{bg13}.
 If $W$ is open, then $\oplam(W)$ is called a model in \cite[Def.~4]{me72}. 
Following \cite{sch98}, we say that $W$ has \textit{almost no boundary} if $\theta_H(\partial W)=0$.
We call $\oplam(W)$ or any of its translates \textit{(measurably) regular} if $W$ has almost no boundary, compare \cite[Def.~7.2]{bg13} and \cite{r07}. For comparison, the so-called regular models of \cite[Def.~2]{me73} require compactness of $W$ in addition. We call $\oplam(W)$ \textit{generic} if $\partial W\cap L^\star=\varnothing$, compare  \cite[Def.~7.2]{bg13} and \cite{sch00}. A generic weak model set can always be obtained from a weak model set $\oplam(W)$ by a suitable shift of the window, compare \cite{bms98}. Following \cite[Def.~4.11]{r07}, we call  $\oplam(W)$ or any of its translates \textit{topologically regular} if $W=\mathrm{cl}(\mathrm{int}(W))$. We call $W$ \textit{aperiodic} if for every $h\in H$ the relation $hW=W$ implies $h=e$, see \cite[Def.~5.12]{r07}. The relevance of these notions for projection sets is discussed in \cite{bms98, sch98, sch00, r07, bg13}.
\end{remark}

\begin{remark}\label{rem:wms}
 Given a non-empty weak model set $\Lambda$, there exist different choices for internal space or a window.
One may restrict to aperiodic $H$ by factoring with the group of periods, see the proof of \cite[Prop.~5.1]{sch98}. If $W$ is a window for $\Lambda$, we always have $W\subset W'$ for some topologically and measurably regular window $W'$, see Remark~\ref{rem:anb}. By \cite[Prop.~5.1]{ms04}, this allows to apply an intrinsic construction of $H$ for $\oplam(W')$ using the so-called autocorrelation topology \cite{bm02}. It can be checked that $H$ is then in fact second countable, since the autocorrelation topology arises from a pseudo-metric. (Compare \cite[Sec.~6,8]{sch98} for a different construction of $H$.) Hence any weak model set admits a second countable internal space. If $W$ is a window for $\Lambda$, then $\Lambda^\star \subset W$.  Often one takes $W \subset \mathrm{cl}(\Lambda^\star)$, but unbounded windows may also be chosen. Compare also the discussion in \cite[Cor.~7]{sch98}.
\end{remark}
\begin{proposition}
Let $\oplam(W)$ be a weak model set with compact window $W$. Then $\oplam(W)$ is relatively dense if and only if\/ $\mathrm{int}(W)\ne\varnothing$.
\end{proposition}

\begin{remark}
The $``\Leftarrow"$ statement is standard. It is a consequence of Lemma~\ref{lem:smallfd} and even holds without the assumption of relative compactness of $W$, see e.g.~\cite[Prop.~3]{me72}, \cite[Prop.~2]{schr73}, \cite[Prop.~2.6]{m97}, \cite[Lem.~2.3 (i)]{st14}, \cite[Prop.~7.5]{bg13}. The $``\Rightarrow"$ statement is a consequence of Proposition~\ref{prop:hrep}. Note that for any (weak) model set $\Lambda$ one may choose the countable set $W=\Lambda^\star$ as a window. Thus $W$ might not have any interior point.
\end{remark}

According to the previous proposition, a compact window $W$ with empty interior implies that $\oplam(W)$ has holes of arbitrary size. In fact, these holes repeat throughout $\oplam(W)$. We say that a uniformly discrete set $\Lambda \subset G$ is \textit{hole-repetitive} if for every compact set $K\subset G$ there is a Delone set $P_K\subset G$ such that $\Lambda \cap t^{-1}K = \varnothing$ for all $t\in P_K$. 
The following statement is proven by an argument from \cite{bm02}. Recall that $W$ is nowhere dense iff $\mathrm{int}(\mathrm{cl}(W))=\varnothing$. 
\begin{proposition}\label{prop:hrep}
Let $(G, H,\mathcal L)$ be a cut-and-project scheme. If $W\subset H$ is relatively compact and nowhere dense, then $\oplam(W)$ is hole-repetitive.
\end{proposition}

\begin{proof}
As $L^\star\subset H$ is countable by Remark~\ref{rem:fdom} (ii) and
$W$ is nowhere dense, the set $L^\star W$ is meagre by definition. Hence, by Baire's category theorem \cite[Thm.~IX.5.1]{bou89}, $L^\star W$ has empty interior. In particular $L^\star W \ne H$, which implies that there
exists $c\in H$ such that $c W\cap L^\star=\varnothing$. Take
arbitrary compact $K\subset G$. Then we have $(K\times c W)\cap \mathcal L=\varnothing$. Take a compact unit neighbourhood
$V\subset H$. Then $K\times (V cW)$ is relatively compact since $W$ is relatively compact. Hence
$(K\times V cW)\cap \mathcal L$ is finite, and we find a unit neighbourhood $U\subset V$ such that $(K\times U cW)\cap \mathcal L=\varnothing$. Now take any $\ell\in L$ from the Delone set $P_K=\oplam(Uc)$. Then $(K\times  \ell W)\cap \mathcal L=\varnothing$, hence $\ell^{-1} K \cap \oplam(W)=\varnothing$. This shows hole-repetitivity.
\end{proof}

\section{The density formula}\label{sec:df}

The density formula expresses the density of a model set $\oplam(W)$ by the volume of its window $W$. For regular model sets it has first been obtained as a consequence of the Poisson summation formula by Meyer for $G\times H=\mathbb R\times\mathbb R^n$,  see \cite{m70}, \cite[Sec.~V.7.3]{me72}, and for $G\times H=\mathbb R^m\times \mathbb R^n$ in \cite[Prop.~5.1]{mm10} and \cite[Lem.~9]{lo13}. This argument can be adapted to general regular model sets. A geometric proof, which ultimately relies on Lemma~\ref{lem:smallfd}, has been given by Schlottmann \cite[Thm.~1]{sch98} for $G=\mathbb R^m$ and general $H$, compare also \cite{ho98}. Proofs via dynamical systems have been given for general $G$ and $H$ by Moody \cite[Thm.~1]{m02} and in \cite[Thm.~9.1]{lr07}. We will refer to \cite{m02} as it fits our needs best.

First we describe suitable averaging sequences.
Consider for $U,W\subset G$ the \textit{(generalised) van Hove boundary}
\begin{displaymath}
\partial^UW=(U\mathrm{cl}(W)\cap \mathrm{cl}(W^c))\cup (U\mathrm{cl}(W^c)\cap \mathrm{cl}(W)),
\end{displaymath}
which was introduced in \cite[Eqn.~(1.1)]{sch00}, see
also \cite[Sec.~2.2]{mr13} for a discussion. The (ordinary) van Hove boundary $U\partial W$ satisfies $U\partial W\subset \partial^UW$. As $\partial W=\partial^{\{e\}}W\subset \partial^U W$ for $U$ any unit neighbourhood, the van Hove boundary may be considered as a thickened topological boundary in that case. A \textit{(generalised) van Hove sequence} is a sequence  $(A_n)_{n\in\mathbb N}$  of compact sets  in $G$ of positive Haar measure $\theta(A_n)$, such that for all compact $K\subset G$ we have
\begin{equation}\label{eq:vh}
\lim_{n\to\infty} \frac{\theta(\partial^K A_n)}{\theta(A_n)}=0.
\end{equation}
Existence of van Hove sequences in $G$ is discussed in \cite{sch00}. In Euclidean space, any sequence of non-empty compact rectangular boxes of diverging inradius is a van Hove sequence. Also any sequence of compact non-empty balls of diverging radius is a van Hove sequence. We  list some properties of van Hove sequences.

\begin{remark}\label{rem:vh}
 Every van Hove sequence is a F{\o}lner sequence \cite{em68}. Consider any van Hove sequence $(A_n)_{n\in\mathbb N}$. Then $\theta(KA_n)=\theta(A_n)+o(\theta(A_n))$ as $n\to\infty$ for any non-empty compact $K$. We also have $\partial^U A_n\subset \partial^V A_n$ for arbitrary $U\subset V$. If $F$ is a compact set containing $e$, then $(FA_n)_{n\in\mathbb N}$ is a van Hove sequence. Indeed, to check the van Hove property one may restrict to compact $K$ containing $e$, and for such $K$ one calculates $\partial^K (FA_n)\subset \partial^{KF} A_n$.
\end{remark}

For a van Hove sequence $(A_n)_{n\in\mathbb N}$ in $G$ and $\xi=(t,h) \in G\times H$, we consider
the relative point frequencies
\begin{equation}\label{eq:freq}
f_n(W,\xi)=\frac{1}{\theta_G(A_n)} \left|t\oplam(h W)\cap A_n\right|,
\end{equation}
where $\left|\cdot\right|$ denotes cardinality, and we also write $f_n(W)$ instead of $f_n(W,e)$, where $e$ is the unit in $G\times H$.

\begin{lemma}[\cite{m02} Density formula for regular model sets]\label{lem:df}
Let $\oplam(W)$ be a regular model set with measurable window
$W$. Then, for every $\xi\in G\times H$ and for every van Hove
sequence, one has
\begin{equation}\label{eq:df}
\lim_{n\to\infty} f_n(W,\xi)=\mathrm{dens}(\mathcal L)\,\theta_H(W).
\end{equation}
\end{lemma}
\begin{remarks}\label{rem:gen}
\item
As remarked in the proof of \cite[Thm.~1]{m02}, for regular model sets the additional Tempelman condition \cite[Eqn.~(5)]{m02} on the van Hove sequence can be dropped. The above limit is independent of the choice of the van Hove sequence. The independence of the choice of $\xi\in G\times H$ is an immediate consequence of translation invariance of $\theta_H$ and $\theta_G$. For metrisable $G$ the convergence is even uniform in $\xi$. This is a consequence of the uniform ergodic theorem, see e.g.~\cite[Thm.~2.16]{mr13}, compare also \cite[Thm.~1]{sch98}.
\item  Assume that in the cut-and-project scheme of Definition~\ref{def:cpscheme}, the direct and internal space are both second countable  locally compact Hausdorff groups and that $\mathcal L$ is a normal discrete subgroup of $G\times H$ such that $(G\times H)/\mathcal L$ is compact. Assume that $G$ admits a van Hove sequence. Inspecting its proof reveals that \cite[Thm.~1]{m02} continues to hold in this case. Indeed, $\mathcal L$ admits relatively compact fundamental domains by \cite{fg68}, and the extended Weil formula, also called the quotient integral formula \cite[Thm~1.5.2]{de09}, remains valid since $G\times H$ is unimodular by \cite[Thm.~9.1.6]{de09}. Since $\mathcal L$ is a normal subgroup, the coset space $(G\times H)/\mathcal L$ carries a canonical group structure, and the induced $G$-action is minimal and hence uniquely ergodic. The uniform ergodic theorem \cite[Thm.~2.16]{mr13} also applies in this situation.
\end{remarks}

As a result of \cite[Thm.~1]{m02},  the density formula~\eqref{eq:df} continues to hold for weak model sets with measurable window for \textit{almost all} $\xi$ within a measurable fundamental domain of the lattice $\mathcal L$. We are interested in an extension to weak model sets which holds for all
$\xi\in G\times H$. A version for Euclidean direct space $G=\mathbb R$ goes back to Meyer \cite[Rem.~(6.2)]{me73}. For $G=\mathbb R^d$ the following proposition appears in Pleasants \cite{p06}, and it is remarked that it can be proved by adaption of Schlottmann's proof of the density formula in \cite{sch98}. 

\begin{proposition}[Density formula for weak model sets]\label{prop:df}
Let $\oplam(W)$ be a weak model set. Then, for every $\xi\in G\times H$ and for
every van Hove sequence, one has
\begin{displaymath}
\mathrm{dens}(\mathcal L)\,\theta_H(\mathrm{int}(W))\le 
\liminf_{n\to\infty} f_n(W,\xi)\le 
\limsup_{n\to\infty} f_n(W,\xi)\le \mathrm{dens}(\mathcal L)\,\theta_H(\mathrm{cl}(W)).
\end{displaymath}
\end{proposition}

\begin{remarks}
\item For regular model sets we get back to the density formula of
  regular model sets (Lemma~\ref{lem:df}). In general neither bound need be attained, since the window may always be chosen countable, or it may be artificially enlarged in the complement of $L^\star$. 
See Section~\ref{sec:sl} for examples of non-regular model sets where the upper bound is sharp.
\item As the following proof does not use commutativity of $G$ or $H$, the conclusion of the above proposition also holds for the non-abelian cut-and-project schemes in Remark~\ref{rem:gen} (ii).
\end{remarks}

Our proof is by approximation with regular model sets.  For preparation, recall that the window of a regular model set has almost no boundary. The following lemma appears, under slightly different assumptions, as a special case in \cite[Lem.~2]{sch98}. Our proof is adapted from \cite[Lem.~6.1]{r07}.

\begin{lemma}\label{lem:vanbd}
Let $G$ be a metrisable locally compact group and fix a left Haar measure for $G$. Consider any neighbourhood $U$ of any compact $K$ in $G$. Then there exists a compact unit neighbourhood $V$ with almost no boundary such that $K\subset KV\subset U$. Moreover, there exists a neighbourhood $W$ of $K$ with almost no boundary such that $K\subset W\subset U$.
\end{lemma}

\begin{proof}
Existence of a unit neighbourhood $V'$ such that $K\subset KV'\subset U$ is standard \cite[Lem.~4.1.3]{de09}.
Fix a left Haar measure $\theta$ and a metric on $G$. Denote by $V_\varepsilon$ the closed ball of radius $\varepsilon$ about the unit element. By choosing $\varepsilon'>0$ sufficiently small, we may assume $V'=V_{\varepsilon'}$ and $V'$ compact. We show that there exists $\varepsilon\in (0,\varepsilon')$ such that $\theta(\partial V_{\varepsilon})=0$, which proves the first claim with the compact unit neighbourhood $V=V_\varepsilon$. Indeed, assume $\theta(\partial V_{\varepsilon})>0$ for all $\varepsilon\in (0,\varepsilon')$. Then there exists $n\in\mathbb N$ such that $I_n=\{\varepsilon\in(0,\varepsilon')\,|\, \theta(\partial V_{\varepsilon})>1/n\}$ is infinite, since otherwise $(0,\varepsilon')=\bigcup_{n\in\mathbb N} I_n$ would be countable. Consider the above $n$ and choose any countably infinite subset $J_n\subset I_n$. We arrive at the contradiction
\begin{displaymath}
\infty > \theta(V_{\varepsilon'})\ge \theta\left(\bigcup_{\varepsilon\in J_n} \partial V_{\varepsilon}\right)=
\sum_{\varepsilon\in J_n}\theta(\partial V_\varepsilon)=\infty.
\end{displaymath}
For the second claim, assume
$K\ne\varnothing$ and take a non-empty open $V'\subset V$. Then the
open cover $(xV')_{x\in K}$ of $K$ contains a finite subcover
$(x_iV')_i$. Then $W=\bigcup_i x_i V\subset U$ is a neighbourhood of
$K$ with almost no boundary since $\partial(\bigcup_i x_i V)\subset
\bigcup_i x_i\partial V$. The case $K=\varnothing$ follows easily.
\end{proof}

\begin{remark}\label{rem:anb}
Lemma~\ref{lem:vanbd} also holds if $G$ is a general LCA Hausdorff group. This can be inferred from the proof of \cite[Lem.~2]{sch98}.
\end{remark}

\begin{proof}[Proof of Proposition~\ref{prop:df}]
By left invariance of the Haar measure on $G\times H$, we assume without loss of generality that $\xi$ is the unit element. Take an arbitrary van Hove sequence.
For the left inequality, assume w.l.o.g.~that
$\mathrm{int}(W)\ne\varnothing$ and fix an arbitrary $\varepsilon>0$. Since $\theta_H$ is inner regular for open sets \cite[Thm.~1.3.4]{de09}, we find compact $K\subset \mathrm{int}(W)$ such that $\theta_H(\mathrm{int}(W))\le \theta_H(K)+\varepsilon/\mathrm{dens}(\mathcal L)$. Choose $U$ with almost no boundary such that $K\subset U \subset \mathrm{int}(W)$ as in Remark~\ref{rem:anb}. Then clearly $f_n(W)\ge f_n(U)$. As $\oplam(U)$ is a regular model set, we have by Lemma~\ref{lem:df}
\begin{displaymath}
\begin{split}
\liminf_{n\to\infty} f_n(W)&\ge \lim_{n\to\infty} f_n(U)=\mathrm{dens}(\mathcal L)\,\theta_H(U)\\
&\ge \mathrm{dens}(\mathcal L)\,\theta_H(K)\ge \mathrm{dens}(\mathcal L)\,\theta_H(\mathrm{int}(W))-\varepsilon.
\end{split}
\end{displaymath}
As $\varepsilon>0$ was arbitrary, we get the claimed lower bound.

For the right inequality, fix arbitrary $\varepsilon>0$. Since
$\theta_H$ is outer regular \cite[Thm.~1.3.4]{de09}, we find a non-empty open $V\supset \mathrm{cl}(W)$ such that $\theta_H(V)\le \theta_H(\mathrm{cl}(W))+\varepsilon/\mathrm{dens}(\mathcal L)$. Choose $U$ with almost no boundary such that $\mathrm{cl}(W)\subset U \subset V$ as in Remark~\ref{rem:anb} above. Then clearly $f_n(W)\le f_n(U)$. As $\oplam(U)$ is a regular model set, we have by Lemma~\ref{lem:df}
\begin{displaymath}
\begin{split}
\limsup_{n\to\infty} f_n(W)&\le \lim_{n\to\infty} f_n(U)=\mathrm{dens}(\mathcal L)\,\theta_H(U)\\
&\le \mathrm{dens}(\mathcal L)\,\theta_H(V)\le \mathrm{dens}(\mathcal L)\,\theta_H(\mathrm{cl}(W))+\varepsilon.
\end{split}
\end{displaymath}
As $\varepsilon>0$ was arbitrary, we get the claimed upper bound.
\end{proof}

\section{Pattern entropy of weak model sets}\label{sec:entms}

\subsection{Complexity measures and pattern entropy}\label{sec:pce}

Let $D$ be a uniformly discrete subset of $G$ and consider a compact
$A\subset G$. Then, for every translation $t\in G$, the finite (possibly empty)
set $D\cap tA$ is called an \textit{$A$-pattern} of $D$.
In order to count $A$-patterns of $D$ up to translation, we recall the complexity measure introduced by Lagarias and Pleasants \cite[Eqn.~(2.8)]{l99}, \cite[Def.~1.4]{lp03}. This has been studied in detail within the class of  regular model sets with convex polyhedral windows  \cite{j10}. 

For counting $A$-patterns, Lagarias and Pleasants consider
\textit{centred} $A$-patterns only, i.e., they assume $e\in A$ and that $A$-patterns are of the form
$D\cap xA$ for $x\in D$ instead of arbitrary $x\in G$. In that case we call $x\in D$ the \textit{center} of the pattern. This gives rise to a complexity measure $N^*_A(D)$ via
\begin{displaymath}
N^*_A(D)=\left|\{x^{-1}D\cap A\,|\, x\in D\}\right|.
\end{displaymath}
Here $\left|\cdot\right|$ denotes cardinality, and the asterisk reminds of centering.
The uniformly discrete set $D$ has \textit{finite local complexity (FLC)} if $N^*_A(D)$ is finite for every compact $A$. In Euclidean space, FLC is equivalent to  $D$ being of finite type \cite[Def.~1.2~(ii)]{l99}. Any discrete FLC set is uniformly discrete. Any weak model set is a uniformly discrete FLC set, since FLC is inherited from the underlying lattice by relative compactness of the window. 

\begin{definition}[Pattern entropy]\label{def:ent}
Let $D\subset G$ be any discrete FLC set, and let $\mathcal A=(A_n)_{n\in\mathbb N}$ be any van Hove sequence in $G$. The \textit{pattern entropy} $h^*_{\mathcal A}(D)\in [0,\infty]$ of $D$ relative to $\mathcal A$ is defined as 
\begin{displaymath}
h^*_{\mathcal A}(D)=\limsup_{n\to\infty}\frac{1}{\theta(A_n)}\log N^*_{A_n}(D),
\end{displaymath}
where $\theta$ is the fixed Haar measure on $G$. We use the convention $\log 0=0$.
\end{definition}

\begin{remarks}\label{rem:cp}
\item
Assume that $G$ is compact. Then $(A_n)_{n\in\mathbb N}$ is a van Hove sequence in $G$ iff $A_n=G$ for eventually all $n$. This follows from $\partial^G A=\varnothing$ if $A\in\{\varnothing, G\}$ and $\partial^G A=G$ otherwise.
Hence in that case $h_{\mathcal A}^*(D)=\theta(G)^{-1}\log N_G^*(D)$ is finite.
\item  Assume that $D$ is an FLC Delone set in Euclidean space and consider a van Hove sequence of balls centred in the origin. Then $D$ has finite pattern entropy \cite[Thm.~2.3]{l99}, which is claimed to be a limit \cite[Eqn.~(1.5)]{lp03}.
\item Weak model sets have finite pattern entropy for every van Hove sequence, see Theorem~\ref{thm:main} for a finite upper bound.
\item The above complexity measure has been introduced to describe discrete FLC sets in Euclidean space which are relatively dense. If $D$ has arbitrarily large holes, then non-centred $A$-patterns may be relevant for complexity, but such patterns are ignored in $N_A^*(D)$. 
\end{remarks}

We discuss another complexity measure which is tailored to general weak model sets. Let $D_0$ be any FLC Delone set of zero pattern entropy and consider $D\subset D_0$. One may identify $D$ with the characteristic function $1_D:D_0\to\{0,1\}$ of $D$, i.e., with the canonical $01$-colouring of $D_0$ induced by $D$.  When counting coloured $A$-patterns in the spirit of Lagarias and Pleasants, this leads to a complexity measure $N^*_A(D, D_0)$ different from $N_A^*(D)$, which is given by
\begin{displaymath}
N^*_A(D, D_0)=\left|\{1_{x^{-1}D\cap A}\,|\, x\in D_0\}\right|,
\end{displaymath}
with characteristic functions $1_{x^{-1}D\cap A}: x^{-1}D_0\cap A\to\{0,1\}$. As in Definition~\ref{def:ent}, we define a corresponding \textit{coloured pattern entropy} $h^*_{\mathcal A}(D,D_0)$ by
\begin{displaymath}
h_{\mathcal A}^*(D,D_0)=\limsup_{n\to\infty}\frac{1}{\theta(A_n)}\log N^*_{A_n}(D, D_0).
\end{displaymath}
The estimate $h_{\mathcal A}^*(D)\le h_{\mathcal A}^*(D,D_0)$ is obvious.
Of particular interest is the case where $D=\oplam(W)$ is a weak model set and where $D_0=\oplam(W_0)$ is a regular model set satisfying $W\subset W_0$. Such $W_0$ always exists due to Remark~\ref{rem:anb}, and $\oplam(W_0)$ has indeed zero pattern entropy by Theorem~\ref{thm:main} below.

\begin{remark}\label{rem:topent}
There are topological dynamical systems naturally associated to $D$ and 01-colourings of $D_0$ induced by $D$, whose topological entropy we denote by $h_t(D)$ and $h_t(D,D_0)$. These have been studied in Euclidean space for a van Hove sequence $\mathcal A$ of balls centred in the origin. If $D$ is an FLC Delone set, then $h_t(D)=h_{\mathcal A}^*(D)$ \cite[Thm.~1]{blr07}. In particular $h_t(D_0)=h_{\mathcal A}^*(D_0)=0$.  Since $D_0$ has zero topological entropy, we have for arbitrary $D\subset D_0$ the result $h_t(D)=h_t(D,D_0)=h_{\mathcal A}^*(D,D_0)$, where the last equality is \cite[Remark~2]{blr07}. This gives a combinatorial method to compute the topological entropy of a weak model set. For a general discrete FLC set $D$, an FLC Delone superset may not exist. But the topological entropy associated to $D$ equals the pattern entropy with respect to the complexity measure $N_A(D)$ where one counts the number of $A$-patterns of $D$ modulo translation by an element of $G$.
\end{remark}

\subsection{The pattern entropy bound}

As a preparation for the following theorem, we study van Hove boundaries.

\begin{lemma}\label{lem:du}
Let $W$ be any relatively compact set in some topological group $G$.   Then $\partial^U W$ is compact for every compact $U\subset G$. If  $G$ is locally compact Hausdorff and $\theta$ is a left Haar measure, then for every $\varepsilon>0$ there exists a compact
unit neighbourhood $U$ such that $\theta(\partial W)\le \theta(\partial^UW)\le\theta(\partial W)+\varepsilon$.
\end{lemma}

\begin{proof}
For compact $U\subset G$ and relatively compact $W\subset G$ the set $U\mathrm{cl}(W)\cap \mathrm{cl}(W^c)$ is compact by continuity of the group multiplication.
Compactness of $U\mathrm{cl}(W^c)\cap \mathrm{cl}(W)$ follows by continuity of the group inversion 
and multiplication from
\begin{displaymath}
U\mathrm{cl}(W^c)\cap \mathrm{cl}(W) = U(\mathrm{cl}(W^c)\cap U^{-1} \mathrm{cl}(W) )\cap \mathrm{cl}(W).
\end{displaymath}
For the second claim, take any unit neighbourhood $U_0$ such that $U_0^{-1}W$ is relatively compact.
Let $\mathcal U_0$ denote the collection of all compact unit neighbourhoods $U\subset U_0$. As $\partial W=\partial^{\{e\}}W\subset \partial^UW$, the second statement clearly follows from
\begin{displaymath}
\inf_{U\in\mathcal U_0} \theta(\partial^UW\setminus \partial W)=0.
\end{displaymath}
A calculation shows that for any $U\in\mathcal U_0$ we have
\begin{displaymath}
\begin{split}
\partial^UW\setminus \partial W&=\left(U\mathrm{cl}(W) \cap \mathrm{int}(W^c)\right) \cup \left(U\mathrm{cl}(W^c) \cap \mathrm{int}(W)\right)\\
&\subset\left(U\mathrm{cl}(W) \cap \mathrm{int}(W^c)\right) \cup \left(U(\mathrm{cl}(W^c)\cap \mathrm{cl}(U_0^{-1}\mathrm{int}(W))) \cap \mathrm{int}(W)\right),
\end{split}
\end{displaymath}
where the sets $\mathrm{cl}(W)$ and $\mathrm{cl}(W^c)\cap \mathrm{cl}(U_0^{-1}\mathrm{int}(W))$ are compact. Now the statement follows from the relation $\inf_{U\in\mathcal U_0}\theta(UK\cap B)=\theta(K\cap B)$, which is valid for compact $K\subset G$ and open $B\subset G$. The latter relation expresses outer regularity of the Haar measure and can be proved similarly to \cite[Lem.~4.1.3]{de09}.
\end{proof}

The following theorem gives pattern entropy estimates for weak model sets.

\begin{theorem}\label{thm:main}
Let $\oplam(W)$ be a weak model set in some cut-and-project scheme $(G, H,\mathcal L)$ with non-compact $G$. Then for any van Hove sequence $\mathcal A$ in $G$ the following pattern entropy estimates hold.
\begin{itemize}
\item[(i)]  For every $\xi\in G\times H$ and for every compact unit neighbourhood $U$ we have
\begin{displaymath}
h_{\mathcal A}^*(\oplam(W))\le \limsup_{n\to\infty} f_n(\partial^UW,\xi)\log 2\le\mathrm{dens}(\mathcal L)\,\theta_H(\partial^UW) \log 2,
\end{displaymath}
where the relative point frequency $f_n(\partial^UW,\xi)$ is defined in~\eqref{eq:freq}.
\item[(ii)] We have the estimate
\begin{displaymath}
h_{\mathcal A}^*(\oplam(W))\le \mathrm{dens}(\mathcal L)\, \theta_H(\partial W) \log 2.
\end{displaymath}
\item[(iii)]  Consider any regular model set $\oplam(W_0)$ such that $W\subset W_0$. Then we have the estimate
\begin{displaymath}
h_{\mathcal A}^*(\oplam(W), \oplam(W_0))\le \mathrm{dens}(\mathcal L)\, \theta_H(\partial W)\log 2,
\end{displaymath}
where $h_{\mathcal A}^*(\oplam(W), \oplam(W_0))$ is the pattern entropy of the $01$-colouring of $\oplam(W_0)$ induced by $\oplam(W)$.
\end{itemize}

\end{theorem}

\begin{remarks}\label{rem:comm}
\item
We infer from Theorem~\ref{thm:main} (ii) the result of \cite[Thm.~5]{blr07} that $W$ with
almost no boundary implies pattern entropy zero. Note that our proof below does not use dynamical systems.
\item A regular model set $\oplam(W_0)$ such that $W\subset W_0$
  always exists due to Remark~\ref{rem:anb}. Hence Theorem~\ref{thm:main} (iii) also gives an upper bound on the topological entropy of the dynamical system associated with $\oplam(W)$, compare Remark~\ref{rem:topent}.
\item The theorem remains true for the non-abelian cut-and-project
  schemes of Remark~\ref{rem:gen} (ii). Indeed, the following proof
  also applies to that situation. In particular, Theorem 12 of
  \cite{em68} yields the desired conclusion since $G$ is unimodular. This holds since $G\times H$ is unimodular by \cite[Thm.~9.1.6]{de09} as it contains the lattice $\mathcal L$.
\end{remarks}

\begin{remark}
The following proof relies on a certain geometric construction which may also be useful in related contexts. We analyse the counting problem in $G\times H$ instead of $G$, as one can then exploit the underlying lattice structure more easily. Patterns centred in $x\in \oplam(W)$ will be shifted such that $(x,e)\in G\times H$ is mapped into a given fundamental domain of the lattice. As a consequence, shifted patterns will be ``properly aligned'' in the fundamental domain, and the counting problem simplifies. 
\end{remark}

\begin{proof}[Proof of Theorem~\ref{thm:main}]
Fix any van Hove sequence $\mathcal A=(A_n)_{n\in\mathbb N}$ in $G$.
We first prove the entropy estimates (i), (ii) which deal with the complexity measure 
$N_A^*(\oplam(W))$. 
Consider any compact unit neighbourhood $U$  in $H$ and choose a
compact $F\subset G$ such that $(F\times U)\mathcal L=G\times H$, see
Lemma~\ref{lem:smallfd}.  Since $\pi_G$ is one-to-one on $\mathcal L$,
we may identify non-empty $A_n$-patterns $xA_n\cap\oplam(W)$ of $\oplam(W)$, where
$x\in G$, with the corresponding
lattice subsets $\pi_G^{-1}(xA_n\cap \oplam(W))= (xA_n\times W)\cap
\mathcal L$. With $(y,u)\in F\times U$ and $\widetilde \ell\in\mathcal
L$ such that $(y,u)=\widetilde \ell \, (x,e)$, we have ${\widetilde
  \ell} \, (xA_n\times W)=yA_n \times uW$. Consequently, $(xA_n\times
W)\cap \mathcal L$ is a lattice translate of
\begin{equation}\label{eq:main}
(yA_n\times uW)\cap \mathcal L =  ((yA_n\times (W\setminus \partial^UW))\cap \mathcal L) \, \dot{\cup}\, \mathcal R_n,
\end{equation}
where $\mathcal R_n=\mathcal R_n(x)$ is some subset of $(yA_n\times \partial^U W)\cap \mathcal L$. This decomposition uses
\begin{displaymath}
\begin{split}
(uW)\setminus (W\setminus \partial^UW)&\subset UW \cap (W^c\cup \partial^UW)\subset \partial^U W
\end{split}
\end{displaymath}
and $W\setminus \partial^UW\subset (U\mathrm{cl}(W^c))^c\subset uW$, which is seen by a similar estimate.
Equation~\eqref{eq:main} tells us that any $A_n$-pattern
$xA_n\cap\oplam(W)$ of $\oplam(W)$ with $x\in G$ appears, up to translation, as some
$A_n$-pattern $yA_n\cap \oplam(W\setminus \partial^UW)$ of $\oplam(W\setminus \partial^UW)$
decorated with some subset of the $A_n$-pattern $yA_n\cap
\oplam(\partial^UW)$ of $\oplam(\partial^UW)$. Here, $y$ is restricted to the compact set $F$.

We now count centred $A_n$-patterns of $\oplam(W)$ by counting the corresponding lattice subsets within $G\times H$. 
The above transformation shifts any pattern center $(x,x^\star)\in\mathcal L$ to the pattern center $\widetilde \ell\, (x,x^\star)=(y,u x^\star)\in (F\times UW)\cap\mathcal L$. Whereas this may give rise to infinitely many values $u=u(x)$, there can only be finitely many values $y=y(x)$ since $F$ is compact. But for any such $y$ we can use the above decomposition to bound the number of $A_n$-patterns $(yA_n\times uW)\cap \mathcal L$ due to different values of $u$. We thus obtain the estimate
\begin{displaymath}
\begin{split}
N^*_{A_n}(\oplam(W)) &\le\left|(F\times U W)\cap\mathcal L\right|
\cdot 2^{\left|(FA_n\times \partial^U W)\cap \mathcal L\right|}.
\end{split}
\end{displaymath}
Since $G$ is not compact, we have $\theta_G(A_n)=\theta_G(A_n^{-1})\to\infty$ as $n\to\infty$ by \cite[Thm.~12]{em68}.  Hence the first factor on the rhs of the above inequality cannot contribute to the pattern entropy. With the convention $\log 0=0$, this leads to
\begin{displaymath}
\frac{h_{\mathcal A}^*(\oplam(W))}{\log 2}\le\limsup_{n\to\infty}\frac{\left|(FA_n\times \partial^U W)\cap \mathcal L\right|}{\theta_G(A_n)} =\limsup_{n\to\infty} f_n(\partial^U W),
\end{displaymath}
where we used Remark~\ref{rem:vh} in the last equation: Since we may assume $e\in F$ without loss of generality,  $(FA_n)_{n\in\mathbb N}$ is a van Hove sequence, and moreover $\theta_G(A_n)=\theta_G(FA_n)+o(\theta_G(A_n))$ as $F$ is compact.  Since $\partial^U W$ is compact by Lemma~\ref{lem:du}, an application of the density formula (Proposition~\ref{prop:df}) for $\oplam(\partial^UW)$ proves (i). 

For (ii) fix arbitrary $\varepsilon>0$. We can apply Lemma~\ref{lem:du} to find a compact unit neighbourhood $U$ such that $\theta_H(\partial^U W)\le\theta_H(\partial W)+\varepsilon$. As $\varepsilon>0$ was arbitrary, claim (ii) follows from (i).

For (iii) consider any regular model set $\oplam(W_0)$ such that $W\subset W_0$. In order to derive the entropy bound for $h_{\mathcal A}^*(\oplam(W),\oplam(W_0))$ we proceed as above and bound the number $N_{A_n}^*(\oplam(W),\oplam(W_0))$. Since we can analyse pattern centers of colour $1$ and of colour $0$ separately, we obtain  
\begin{displaymath}
\begin{split}
N^*_{A_n}(\oplam(W),\,&\oplam(W_0)) \le \left|F\cap\oplam( UW)\right|\cdot 2^{\left|FA_n\cap\oplam(\partial^U W)\right|+\left|FA_n\cap\oplam( \partial^U W_0)\right|}\\
&+\left|F\cap \oplam(UW_0)\right|\cdot 2^{\left|FA_n\cap\oplam(\partial^U W)\right|+\left|FA_n\cap\oplam(\partial^U W_0)\right|}.
\end{split}
\end{displaymath}
Since $W_0$ has almost no boundary, we obtain the same estimate as in (ii) above.
\end{proof}

\section{Subsets of a lattice}\label{sec:sl}

We discuss a cut-and-project scheme that is naturally associated to
subsets of the lattice $\mathbb Z^n\subset \mathbb R^n$. Whereas an
adelic version already appeared in \cite{me73, bmp00, bm02}, we
use the simpler formulation from \cite[Sec.~5a]{sing07} which is sufficient for our needs.

\subsection{The cut-and-project scheme}

For fixed $k\in\mathbb N$ consider $H=\prod_{p} (\mathbb Z^n/ p^k\mathbb Z^n)$, where here
and below $p$
runs through all primes, 
together with the canonical group structure inherited from its
factors. We equip $H$ with the product topology with respect to the
discrete topology on the finite quotient groups $\mathbb Z^n/ p^k\mathbb
Z^n$ of order $p^{nk}$. Thus $H$ is a second countable compact abelian group. We choose as Haar
measure on $H$ the product measure inherited from its factors.
Define a star-map $\star:\mathbb Z^n\to H$ by
\begin{displaymath}
x \mapsto x^\star=(x\mod p^k\mathbb Z^n)_p,
\end{displaymath}
where $x\mod p^k\mathbb Z^n$ is the image of $x$ under the canonical
projection $\pi_p\!:\,\mathbb Z^n\rightarrow \mathbb Z^n/p^k\mathbb
Z^n$ onto the quotient group.
It is not difficult to show that $(\mathbb Z^n, H, \iota(\mathbb
Z^n))$, where $\iota\!:\, \mathbb Z^n\rightarrow \mathbb Z^n\times H$
is the ``diagonal'' embedding $x\mapsto (x,x^\star)$, is
a cut-and-project scheme. As the restriction $\pi_H|_{\mathcal L}$ resp.\ the star-map are one-to-one, every subset of $\mathbb Z^n$ is a weak model set with the same internal space.

\subsection{$k$-free lattice points}

The \textit{$k$-free lattice points} $V(k,n)$ of $\mathbb Z^n$ \cite{ph13} are
given by
$$
V(k,n)=\mathbb Z^n\setminus\bigcup_p p^k\mathbb Z^n.
$$ 
To avoid the trivial case $V(1,1)=\{\pm1\}$, we assume that $nk>1$. The set $V(1,n)$ is called the \textit{visible lattice points}. Let us mention that the \textit{squarefree numbers} $V(2,1)$ 
have been constructed as a projection set already in \cite{me73}, where it was also proven that $V(k,1)$ has holes of arbitrary size \cite[Lem.~1]{me73}. In fact the $k$-free lattice
points are weak model sets. Indeed, with respect to the above
cut-and-project scheme one has
$V(k,n)=\oplam(W)$, where 
$$
W=\prod_p (\mathbb Z^n/ p^k\mathbb Z^n)\setminus\{0\!\!\!\!\mod p^k\mathbb Z^n\}.
$$
Since no component of $W$ is maximal, we also have $\mathrm{int}(W)=\varnothing$. Note that $W$ is closed as every component of $W$ is closed. We conclude that $W$ is nowhere dense.
Consequently,  the set of $k$-free lattice points is hole-repetitive
by Proposition~\ref{prop:hrep}. In fact, holes repeat lattice
periodically by the Chinese remainder theorem; see the proof of
\cite[Prop.~1]{ph13}. Also $\mathrm{cl}(W)=W$ since every component is closed. This means that $W=\partial W$. 
The Haar measure of $W$ is given by
\begin{displaymath}
\theta_H(W)=\prod_{p} \left(1-\frac{1}{p^{nk}}\right)=\frac{1}{\zeta(nk)}>0,
\end{displaymath}
which coincides with the density of the $k$-free lattice points \cite{bmp00,ph13}. For $k$-free lattice points and for any centred van Hove sequence $\mathcal A$ of cubes, also the upper bound in the estimate of $h_{\mathcal A}^*(\oplam(W),\mathbb Z^n)$
is attained: Let $A$ be any compact set in $\mathbb Z^n$. Noting
that for these models every subset of a coloured $A$-pattern is a coloured $A$-pattern
\cite[Thm.~2]{ph13}, we clearly have
\begin{displaymath}
N^*_A(\oplam(W),\mathbb Z^n)\ge 2^{|(A\times W)\cap \mathcal L|}.
\end{displaymath}
The claim follows by noting that the latter estimate implies
\begin{displaymath}
h^*_{\mathcal A}(\oplam(W),\mathbb Z^n)\ge \underline{\mathrm{dens}}(V(k,n)) \log 2=\theta_H(W)\log2,
\end{displaymath}
where the last equality uses that the density exists and is given by the Haar measure of the window. 

\subsection{Complementary lattice subsets}

Let $S\subset\mathbb Z^n$ be any lattice subset and $S^c=\mathbb Z^n\setminus S$ be its lattice complement. Then both sets have the same pattern entropy, i.e., $h_{\mathcal A}^*(S,\mathbb Z^n)=h^*_{\mathcal A}(S^c,\mathbb Z^n)$ for every centred van Hove sequence $\mathcal A$.
This is clear after identifying $S$ with the $01$-colouring of $\mathbb Z^n$ induced by $S$. The points of the complementary lattice subset are then obtained by colour inversion. But inverting colours does not affect pattern counting. Hence the pattern entropies coincide. An example is given by the set $S^c$ of \textit{invisible lattice points} \cite{b02, bg13} in $\mathbb Z^n$ which is the lattice complement of the set $S=V(1,n)$  of visible lattice points. The set of invisible lattice points is a model set, the window being the complement of the window of
visible lattice points. Hence the set of invisible lattice points is
Delone, since the window is non-empty, open and relatively compact. As argued
above, the pattern entropy coincides with that of the visible lattice
points, and the pattern entropy bound of Theorem~\ref{thm:main} (iii)
is sharp.

\section*{Acknowledgements}

This work was initiated during a visit of one of the authors to Bielefeld University in spring 2014. Support by the German Research Council (DFG) within the CRC 701 ``Spectral Structures and Topological Methods in Mathematics'' is gratefully acknowledged. Results of this paper have been presented \cite{r14} during the Mini-Workshop ``Dynamical versus Diffraction Spectra in the Theory of Quasicrystals'' at the MFO in Oberwolfach in December 2014. We thank the participants for discussions and Michael Baake for helpful comments on the manuscript. We also thank the referee for a careful reading, comments and a correction.

\end{document}